\documentclass[11pt]{article}
\usepackage{amsmath,amsthm,amssymb,ifsym,a4wide} 
\usepackage[twoside]{geometry}
\usepackage{graphicx}
\usepackage{epsfig} 
\usepackage{slashed}  
\usepackage[hidelinks]{hyperref}
\newtheorem{theorem}{Theorem}[section]
\newtheorem{proposition}[theorem]{Proposition}
\newtheorem{lemma}[theorem]{Lemma}

\newtheorem{definition}[theorem]{Definition}
 
\numberwithin{equation}{section} 
\numberwithin{figure}{section}  
\usepackage{amssymb, amscd, mathrsfs, wasysym} 

\newcommand \la \langle
\newcommand \ra \rangle

\newcommand \Ecal {\mathcal{E}} 
\newcommand \Kcal {\mathcal{K}}
\newcommand \Hcal {\mathcal{H}}

\newcommand \underdel {\underline \partial}

\newcommand \trianglerightNEW \triangleright

\newcommand \auth {\textsc}

\newcommand \bei {\begin{itemize}}
\newcommand \eei {\end{itemize}}
\newcommand \be {\begin{equation}}
\newcommand \bel {\be\label}
\newcommand \ee {\end{equation}}
\newcommand \del \partial
\newcommand \RR {\mathbb R}

\newcommand \eps \epsilon

\usepackage[most]{tcolorbox} 

\let\oldmarginpar\marginpar
\renewcommand\marginpar[1]{\-\oldmarginpar[\raggedleft\footnotesize #1]%
{\raggedright\footnotesize #1}}
\begin{document}

\title{\bf \Large 
Global solution to the Klein-Gordon-Zakharov equations with uniform energy bounds}

\author{Shijie Dong\footnote{Fudan University, School of Mathematical Sciences, 220 Handan Road, Shanghai, 200433, China.
Email: shijiedong1991@hotmail.com.
}
}


\date{\today}
\maketitle

\begin{abstract} 
{
We are interested in the Klein-Gordon-Zakharov equations in $\mathbb{R}^{1+3}$, and we aim 
to show that the energy for the global solution to the equations is uniformly bounded, and we do not require the compactness assumption on the initial data. To achieve these goals, the key is to apply Alinhac's ghost weight energy estimates adapted to the Klein-Gordon equations.
}
\end{abstract}

\tableofcontents

\section{Introduction}

\subsection{Model problem and main result}

We are interested in the Klein-Gordon-Zakharov model in $\RR^{1+3}$, which describes the nonlinear interaction between Langmuir waves and ion-acoustic waves. The equations read
\bel{eq:model-KGZ}
\aligned
-\Box E + E = - n E,
\\
-\Box n = \Delta |E|^2,
\endaligned
\ee
and the initial data are prescribed on the slice $t=t_0 =0$
\bel{eq:model-ID}
\aligned
\big( E, \del_t E \big) (t_0, \cdot)
=
\big( E_0, E_1 \big),
\qquad
\big( n, \del_t n \big) (t_0, \cdot)
=
\big( n_0, n_1 \big).
\endaligned
\ee
The unknowns include $E = (E^1, E^2, E^3)$ taking values in $\RR^3$ and $n$ taking values in $\RR$, which can be regarded as a Klein-Gordon component and a wave component respectively.
As usual, the wave operator is denoted by $\Box = \del_\alpha \del^\alpha = -\del_t \del_t + \Delta$, with the Laplace operator $\Delta = \del_a \del^a$. 
We take the signature $(-, +, +, +)$ in the spacetime $\RR^{1+3}$, and the indices are raised or lowered by the Minkowski metric $m = \text{diag} (-1, 1, 1, 1)$. The Greek letters $\alpha, \beta, \cdots \in \{0, 1, 2, 3\}$ are used to denote the spacetime indices, and Latin letters $a, b, \cdots \in \{1, 2, 3\}$ are used to represent the space indices, and we adopt the Einstein summation convention. The notation $\lesssim$ will be frequently used, and we write $A\lesssim B$ to indicate $A \leq C B$ with $C$  a universal constant. The Japanese bracket $\langle A \rangle = \sqrt{1+|A|^2}$ is used in the usual way, and we use the abbreviation $\| \cdot \| = \| \cdot \|_{L^2}$ to denote the $L^2$ norm.


In addition to showing the small data global existence result to the model problem \eqref{eq:model-KGZ}, and to deriving the pointwise asymptotic behavior of the solution, it is also interesting to verify whether the energy for the solution is uniformly bounded or blows up at infinity. We now provide the statement of the main result.

\begin{theorem}\label{thm:main1}
Consider the Klein-Gordon-Zakharov equations in \eqref{eq:model-KGZ}, and let $N \geq 15$ be an integer. There exits $\eps_0 >0$, such that for all $\eps < \eps_0$, and all initial data satisfying the smallness condition
\be 
\sum_{|I| \leq N+2} \| \langle x \rangle^{|I| + 1} \del^I E_0 \| + \sum_{|I| \leq N+1} \| \langle x \rangle^{|I| + 2} \del^I E_1 \|
+
\sum_{|I| \leq N+1} \| \langle x \rangle^{|I| } n_0 \| + \sum_{|I| \leq N} \| \langle x \rangle^{|I| + 1} n_1 \|
\leq \eps,
\ee
the Cauchy problem \eqref{eq:model-KGZ}--\eqref{eq:model-ID} admits a global-in-time solution $(E, n)$, 
which satisfies the following sharp pointwise decay results
\bel{eq:thm-decay}
|E(t, x)| \lesssim \langle t \rangle^{-3/2},
\quad
|n(t, x) | \lesssim \langle t \rangle^{-1} \langle t-r\rangle^{-1/2}.
\ee
With $\Gamma \in \{ \del_\alpha, L_a, \Omega_{ab} \}$, the following uniform energy estimates are also valid
\bel{eq:thm-E}
\aligned
\| \del \Gamma^I E(t) \|+ \| \Gamma^I E(t) \|
\lesssim 1,
\qquad
|I| \leq N+1,
\\
\| \Gamma^I n(t) \|
\lesssim 1,
\qquad
|I| \leq N+1.
\endaligned
\ee

\end{theorem}

In terms of small data global existence result for the Klein-Gordon-Zakharov equations and the asymptotic behavior of the solution, there exist a few proofs, see for instance \cite{OTT, Tsutaya, Katayama12a, Dong1811}. 
So the main contribution in Theorem \ref{thm:main1} is that the energy for the solution to the Klein-Gordon-Zakharov equations is shown to be uniformly bounded.

In order to prove the energy for $E, n$ is uniformly bounded, the key is to apply Alinhac's ghost weight method adapted to Klein-Gordon equations, which was used in \cite{Dong2005} by the author when dealing with a coupled wave and Klein-Gordon system in $\RR^{1+2}$. Originally, the ghost weight method was applied to wave equations with null nonlinearities \cite{Alinhac1}, and allows one to benefit from the $\langle t-r \rangle$ decay when some good derivatives $(x_a/r) \del_t + \del_a$ acting on the wave components. But we find in \cite{Dong2005} that the ghost weight energy estimates on Klein-Gordon equations provide us with some new and strong result, i.e. we can benefit from the $\langle t-r \rangle$ decay for the Klein-Gordon components without any derivatives (or with good derivatives $(x_a/r) \del_t + \del_a$). 
The analysis in the proof of Theorem \ref{thm:main1} can also be expected to treat more general systems of coupled wave and Klein-Gordon equations.

\subsection{Brief history and motivation} 

We first very briefly recall the pioneering work in the study of pure wave and pure Klein-Gordon equations in $\RR^{1+3}$.
The fact that wave equations with quadratic null nonlinearities admit global-in-time solutions was first proved in the seminal work \cite{Klainerman86} by Klainerman and \cite{Christodoulou} by Christodoulou. On the other hand, in the breakthrough by Klainerman \cite{Klainerman85} and Shatah \cite{Shatah} the Klein-Gordon equations with quadratic nonlinearities were shown to possess small global solutions.

Inspired by the aforementioned pioneering work on pure wave and pure Klein-Gordon equations in $\RR^{1+3}$, the coupled wave and Klein-Gordon systems have been actively studied since decades ago. 
Such studies are motivated by the interest of understanding general PDE's, or are motivated by wave and Klein-Gordon systems derived from physical models.

To our knowledge, the first study on this subject was due to Bachelot \cite{Bachelot} on the Dirac-Klein-Gordon equations. After that, Georgiev \cite{Georgiev} proved the global existence result for strong null nonlinearities (i.e. $\del_\alpha u \del_\beta v - \del_\beta u \del_\alpha v$, $\alpha, \beta \in \{0, 1, 2, 3\}$). Later on, more physical models governed by the wave and Klein-Gordon systems were studied, for example the Klein-Gordon-Zakharov equations \cite{OTT, Tsutaya, Katayama12a, Dong1811}, the Maxwell-Klein-Gordon equations \cite{PsarelliA, PsarelliT, Klainerman-QW-SY, Fang}, the Dirac-Proca equations \cite{Tsutsumi}, the Einstein-Klein-Gordon equations \cite{PLF-YM-cmp, PLF-YM-arXiv1, Wang, Ionescu-P, Ionescu-P2}, and the $U(1)$ electroweak standard model \cite{DLW}. Besides, in view of pure PDE's, there also exist some other results, see for instance \cite{Katayama12a, PLF-YM-book, Katayama18, Dong1811, Dong1912}.

Back to the Klein-Gordon-Zakharov equations appearing in plasma physics, one refers to its physical background in \cite{Zakharov, Dendy}. Besides of the existence of small global solution to the equations in $\RR^{1+3}$, there exist many other interesting results concerning different aspects of the equations. In \cite{OTT2}, Ozawa, Tsutaya, and Tsutsumi proved that the Klein-Gordon-Zakharov equations with different propagation speeds admit global solution in $\RR^{1+3}$. Later on, Masmoudi and Nakanishi \cite{Masmoudi} showed the convergence of the Klein–Gordon–Zakharov equations to the Schrodinger equation when certain parameters go to $+\infty$. Recently in the work of Shi and Wang \cite{WangShu}, the authors studied the finite time blow-up result for the Klein-Gordon-Zakharov equations with very low regularity on the initial data and with negative initial energy. Equally importantly, it is worth to mention that the Klein-Gordon-Zakharov equations in $\RR^{1+2}$ also admit global solutions and enjoy pointwise decay properties for the small solutions, see \cite{Guo, Dong2006, Ma2008, Duan-Ma}.

Recently, based on the work of Klainerman \cite{Klainerman86} and Hormander \cite{Hormander}, the hyperboloidal foliation method (essentially the vector field method on hyperboloids) was developed to deal with coupled wave and Klein-Gordon systems in \cite{PLF-YM-book, PLF-YM-cmp}. In this method, one of the key features is that we can take advantage of the $(t-r)$ decay in the nonlinearities. In order to better illustrate this, we provide with a simple example here, and more detailed discussion can be found in Section \ref{sec:Comparison}. We assume $u, v$ are solutions to homogeneous wave and Klein-Gordon equations (with nice initial data for simplicity), i.e.
$$
-\Box u = 0,
\qquad
-\Box v + v = 0,
$$ 
and thus it is reasonable to assume
$$
\aligned
&\| \del u \| \lesssim 1,
\qquad\qquad
&|\del u | \lesssim \langle t+r \rangle^{-1} \langle t-r\rangle^{-1/2},
\\
&\| \del v \| + \| v \| \lesssim 1,
\qquad\qquad
&|\del v | \lesssim \langle t+r \rangle^{-3/2}.
\endaligned
$$
On one hand, if we use the usual flat foliation (i.e. the constant $t$ foliation) of the spacetime, naive calculations give us
$$
\aligned
\int_{t_0}^t \|\del u v \| \, dt
\leq
\int_{t_0}^t \|\del u\|_{L^\infty} \|v \| + \|\del u \| \|v \|_{L^\infty} \, dt'
\lesssim
\int_{t_0}^t (\langle t'\rangle^{-1} + \langle t'\rangle^{-3/2}) \, dt'
\lesssim \log (1+t),
\endaligned
$$
and in the above we note that the $(t-r)$ decay is lost when we take sup-norm $\| \langle t+r \rangle^{-1} \langle t-r \rangle^{-1/2} \|_{L^\infty} \lesssim t^{-1}$.
On the other hand, the application of the hyperboloidal foliation method provides us with
$$
\aligned
&\int_{s_0}^s \|\del u v \|_{L^2_f(\Hcal_{s'})} \, ds'
\\
\leq
&\int_{s_0}^s \| (s'/t) \del u\|_{L^2_f(\Hcal_{s'})} \|(t/s') v \|_{L^\infty(\Hcal_{s'})} + \|\del u \|_{L^\infty(\Hcal_{s'})} \|v \|_{L^2_f(\Hcal_{s'})} \, dt'
\lesssim
\int_{s_0}^s (s'^{-3/2}) \, ds'
\lesssim 1,
\endaligned
$$
with the hyperbolic time
$$
s = \sqrt{t^2 - r^2} \leq t,
$$
and under reasonable assumptions\footnote{In fact, we also assume that $s_0\geq 2$ and the solutions $u, v$ are supported in $\{ (t, x): t \geq 2, t \geq r+1 \}$, and the notations $\Hcal_{s}, \| \cdot \|_{L^2_f(\Hcal_{s})}$ will be introduced in Section \ref{sec:Comparison}.} that
$$
\aligned
&\| (s/t) \del u \|_{L^2_f(\Hcal_{s})} \lesssim 1,
\qquad\qquad
&|\del u | \lesssim \langle t+r \rangle^{-1} \langle t-r\rangle^{-1/2} \lesssim s^{-3/2},
\\
&\| (s/t) \del v \|_{L^2_f(\Hcal_{s})} + \| v \|_{L^2_f(\Hcal_{s})} \lesssim 1,
\qquad\qquad
&|\del v | \lesssim \langle t+r \rangle^{-3/2} \lesssim s^{-3/2}.
\endaligned
$$
When using the hyperboloidal foliation method, we integrate over the new time variable $s \simeq t^{1/2} (t-r)^{1/2}$, which makes use of the $(t-r)$ decay, and this gives us uniform bound of the spacetime integral. Details can be found in Section \ref{sec:Comparison}, and one can also refer to \cite{PLF-YM-book}.

Inspired by the usage of the $(t-r)$ decay in the hyperboloidal foliation method, it is natural to ask whether one can take advantage of the $(t-r)$ decay in the constant $t$ foliation of the spacetime. We find that this can be achieved by applying Alinhac's ghost weight method adapted to Klein-Gordon equations. Recall in \cite{Dong2005} we showed that the following estimates can be derived using the ghost weight method (details can be found in Proposition \ref{prop:gst} or in \cite{Dong2005})
$$
\int_{t_0}^t \Big\| {v \over \langle t'-r \rangle^{1/2+\delta} } \Big\|^2 \, dt'
\lesssim 1,
\qquad
\delta > 0,
$$
which also lead us to uniform bounds of the spacetime integral
$$
\aligned
&\quad
\int_{t_0}^t \|\del u v \| \, dt
\\
&\leq
\int_{t_0}^t \|\langle t'-r \rangle^{1/2+\delta} \del u\|_{L^\infty} \|\langle t'-r \rangle^{-1/2-\delta} v \| + \|\del u \| \|v \|_{L^\infty} \, dt'
\\
&\lesssim
\Big(\int_{t_0}^t \|\langle t'-r \rangle^{1/2+\delta} \del u\|^2_{L^\infty} \, dt'\Big)^{1/2} \Big(\int_{t_0}^t \|\langle t'-r \rangle^{-1/2-\delta} v \|^2 \, dt'\Big)^{1/2}
+
\int_{t_0}^t \langle t' \rangle^{-3/2} \, dt'
\\
&\lesssim
\Big(\int_{t_0}^t \langle t' \rangle^{2(-1+\delta)} \, dt'\Big)^{1/2} +1 
\lesssim 1.
\endaligned
$$

The key novelty of this paper is to show the uniform energy bounds for the Klein-Gordon-Zakharov equations by applying Alinhac's ghost weight method to Klein-Gordon equations. As we have demonstrated in the simple calculations above, it is quite necessary to gain some $t-r$ decay in the nonlinearities, and this is also highly non-trivial in the flat foliation case (while the $t-r$ decay can be easily obtained from the energy estimates and the Klainerman-Sobolev inequalities in the hyperboloidal foliation case, see for instance \cite{PLF-YM-book}). We conquer this difficulty by a novel use of the scaling vector field as in \cite{Dong1912, Dong2005}, which is generally avoided in the study of coupled wave and Klein-Gordon equations.

\subsection{Reformulation of the Klein-Gordon-Zakharov equations}

In the original formulation of the Klein-Gordon-Zakharov equations, there appears wave component without derivatives in the nonlinearities. This will somewhat cause difficulties because the wave component without derivatives cannot be controlled by its nature energy. So we reformulate the equations as
\bel{eq:model-re}
\aligned
&-\Box E + E = -n^0 E - \Delta n^\Delta E,
\\
&-\Box n^0 = 0,
\qquad
-\Box n^\Delta = |E|^2,
\endaligned
\ee
with initial data on $t=t_0$
\bel{eq:ID-re}
\big(E, \del_t E, n^0, \del_t n^0, n^\Delta, \del_t n^\Delta \big)(t_0)
=
(E_0, E_1, n_0, n_1, 0, 0).
\ee
The following relation can be used to estimate the original unknowns $(E, n)$ by the new unknowns $(E, n^0, n^\Delta)$
\bel{eq:relation}
n = n^0 + \Delta n^\Delta.
\ee
This kind of reformulation was used by Katayama \cite{Katayama12a} for wave equations with nonlinearities of divergence form, which provides one with better estimates for the wave components without derivatives.
In the new formulation \eqref{eq:model-re}, we note that the wave component $n^0$ is a free wave, while the wave component $n^\Delta$ is with derivatives in the nonlinearities, and thus we expect that the reformulated system of equations is easier to handle.

\subsection{Some unsolved problems}

Besides of the Klein-Gordon-Zakharov model, there exist many other physical models which can be described by coupled wave and Klein-Gordon equations, including the Dirac-Klein-Gordon model, the Maxwell-Klein-Gordon model, the Dirac-Proca model, the Einstein-Klein-Gordon model, etc.. The existence of small global solutions and pointwise decay of the solutions for the aforementioned models in $\RR^{1+3}$ have already been obtained. But whether the energy for the solutions is uniformly bounded is still unknown, except the Klein-Gordon-Zakharov model.

We take the (simplified) Einstein-Klein-Gordon equations as an example, and we recall that there appear terms of type $u \del \del v$ in the nonlinearities, with $u$ a wave component, and $v$ a Klein-Gordon component. But in that case, it seems that the best decay rate we can get for the wave component is $|u| \lesssim t^{-1}$, which does not have any extra $t-r$ decay. This observation unfortunately indicates that there is little chance for one to obtain the uniform energy bounds for the Einstein-Klein-Gordon equations by applying either the ghost weight method or the hyperboloidal foliation method.


\subsection*{Outline}

The rest of this article is organised as follows.

In Section \ref{sec:BHFM}, we revisit some preliminaries on the wave equations as well as the vector field method. 
Then in Section \ref{sec:Comparison}, we provide some examples to make a comparison between the ghost weight method and the hyperboloidal foliation method.
Finally we prove Theorem \ref{thm:main1} relying on the contraction mapping argument in Section \ref{sec:Contraction}.


\section{Preliminaries}
\label{sec:BHFM}
 
\subsection{Basic notations} 
 
We work in the $(1+3)$ dimensional Minkowski spacetime with metric $m = \text{diag} (-1, 1, 1, 1)$, which is used to raise or lower the spacetime indices $\alpha, \beta, \cdots \in \{0, 1, 2, 3\}$ and the space indices $a, b, \cdots \in \{1, 2, 3\}$. We denote one point by $(x^0, x^1, x^2, x^3) = (t, x^1, x^2, x^3)$ with its spacial radius $r = \sqrt{(x_1)^2 + (x_2)^2 + (x_3)^2}$.
 
The following vector fields will be used in the analysis
\bei
\item Translations: $\del_\alpha = \del_{x^\alpha}$.

\item Lorentz boosts: $L_a = x_a \del_t + t \del_a$.

\item Rotations: $\Omega_{ab} = x_a\del_b - x_b\del_a$.

\item Scaling vector field: $L_0 = S = t \del_t + r \del_r$.

\eei

We will use $\Gamma$ to denote a general vector filed in the set $V = \{\del_\alpha, L_a, \Omega_{ab} \}$. We recall the relation
\bel{eq:scaling}
\langle t-r \rangle |\del w| 
\lesssim |\Gamma w| + |L_0 w|,
\ee
which can give us extra $t-r$ decay for $\del w$.

\subsection{Energy estimates for wave and Klein-Gordon equations}

We will provide two kinds of energy estimates for Klein-Gordon equations (i.e. the natural energy and the ghost weight energy) and three kinds of energy estimates for wave equations (i.e. the natural energy, the ghost weight energy, and the conformal energy).

\paragraph{Energy estimates for wave and Klein-Gordon equations}

Let $v$ be the solution to
$$
- \Box v + m^2 v = f,
\qquad
\big( v, \del_t v \big)(t_0) = (v_0, v_1),
\qquad
m \geq 0,
$$
The natural energy estimates for wave and Klein-Gordon equations are well-known, which read
\bel{eq:NEE}
\Ecal_m(t, v)^{1/2}
\lesssim
\Ecal_m(t_0, v)^{1/2}
+
\int_{t_0}^t \|f \| \, dt',
\ee
with
\be 
\Ecal_m (t, v)
=
\int_{\RR^3} |\del_t v|^2 + \sum_a |\del_a v|^2 + m^2 v^2 \, dx.
\ee
For simplicity, we will use the abbreviation $\Ecal(t, v) = \Ecal_0 (t, v)$.

\paragraph{Ghost weight energy estimates for wave and Klein-Gordon equations}

We now recall the ghost weight energy estimates adapted to Klein-Gordon equations, which was used in \cite{Dong2005}.

\begin{proposition}\label{prop:gst}
Assume $v$ is the solution to 
$$
- \Box v + m^2 v = f,
\qquad
m \geq 0,
$$
then it holds
\bel{eq:GEE} 
\aligned
&\Ecal_{gst, m} (t, v)
\lesssim
\int_{\RR^3} \big( |\del_t v|^2 + \sum_a |\del_a v|^2 + m^2 v^2 \big) \, dx (t_0)
+
\int_{t_0}^t \int_{\RR^3}\big| f \del_t v \big| \, dxdt',
\endaligned
\ee
in which (with $\delta > 0$) \footnote{We use the abbreviation $\Ecal_{gst} (t, v) = \Ecal_{gst, 0} (t, v)$, and in the proof we will take $\delta\ll 1$.}
\be 
\aligned
\Ecal_{gst, m} (t, v)
=
&\int_{\RR^3} \big( |\del_t v|^2 + \sum_a |\del_a v|^2 + m^2 v^2 \big) \, dx 
+
m^2 \int_{t_0}^t \int_{\RR^3} {v^2 \over \langle r-t' \rangle^{1+\delta}}  \, dxdt'
\\
+
& \sum_{a} \int_{t_0}^t \int_{\RR^3} {1 \over \langle r-t' \rangle^{1+\delta}} \big| G_a v \big|^2 \, dxdt'.
\endaligned
\ee
\end{proposition}

\begin{proof}
For completeness, we revisit its proof, which is almost the same as the proof for the case of wave equations.

We multiply on both sides of the $v$ equation with $e^q \del_t v$ to get
$$
\aligned
&{1\over 2} \del_t \big( e^q (\del v)^2 + m^2 e^q v^2  \big)
-
\del_a \big( e^q \del^a v \del_t v \big)
+
{1\over 2} {e^q \over \langle t-r \rangle^{1+\delta}} \sum_a \big( G_a v \big)^2
\\
+
&{m^2\over 2} {e^q \over \langle t-r \rangle^{1+\delta}} v^2
=
e^q f \del_t v,
\endaligned
$$
in which 
$$
q = \int_{-\infty}^{r-t} \langle s \rangle^{-1 - \delta} \, ds,
\qquad
\delta > 0.
$$

Integrating the differential equality over the region $[t_0, t] \times \RR^3$ and noting $0\leq q \lesssim 1$ lead us to the desired energy estimates. 
The proof is done.
\end{proof}

\paragraph{Conformal energy estimates for wave equations}

\begin{proposition}
Let $u$ be the solution to
$$
-\Box u = f,
\qquad
\big( u, \del_t u \big) (t_0) = (u_0, u_1),
$$
then we have
\bel{eq:conformal-EE}
\Ecal_{con} (t, u)^{1/2}
\lesssim
\Ecal_{con} (t_0, u)^{1/2}
+
\int_{t_0}^t \big\| \langle t'+|x| \rangle f \big\| \, dt',
\ee
in which
\be 
\Ecal_{con} (t, u)
=
\| u \|^2 + \| L_0 u \|^2 + \sum_{a<b} \| \Omega_{ab} u \|^2 + \sum_a \| L_a u \|^2. 
\ee
\end{proposition}

\subsection{Sobolev--type inequalities}

In order to obtain the pointwise wave decay or Klein-Gorodn decay estimates from the weighted energy bounds we recall the following two kinds of inequalities. The importance of the inequalities below, to the study of the coupled wave and Klein-Gordon equations, is that we do not need to rely on the scaling vector field $L_0 = t \del_t + x^a \del_a$.

We first revisit one special version of the Klainerman-Sobolev inequality in \cite{Klainerman852}, see the inequalities (4), (5'), and (6) therein (or Theorem 2 in \cite{Klainerman87}). Since we can avoid the use of the scaling vector field $L_0$, so this very inequality is particularly well adapted to the study of the coupled wave and Klein-Gordon systems. When deriving the pointwise bounds for a given function at time $t$ using the inequality \eqref{eq:K-S}, we need the future information of the function till time $2t$, so we rely on the contraction mapping method to prove Theorem \ref{thm:main1}.

\begin{proposition}\label{prop:K-S}
Let $u = u(t, x)$ be a sufficiently smooth function which decays sufficiently fast at space infinity for each fixed $t \geq 0$.
Then for any $t \geq 0$, $x \in \RR^3$, we have
\bel{eq:K-S}
|u(t, x)|
\lesssim \langle t + r \rangle^{-1} \sup_{0\leq s \leq 2t, |I| \leq 3} \big\| \Gamma^I u(s) \big\|,
\qquad
\Gamma \in V = \{ L_a, \del_\alpha, \Omega_{ab} \}.
\ee
\end{proposition}

Next, we introduce some notations in order to state the inequality for obtaining pointwise estimates for the Klein-Gordon components, which was proved by Georgiev in \cite{Georgiev2}. 
Let $\{ p_j \}_0^\infty$ be a usual Paley-Littlewood partition of the unity
$$
1 = \sum_{j \geq 0} p_j(s),
\qquad
s \geq 0,
$$
which also satisfies 
$$
0 \leq p_j \leq 1,
\qquad
p_j \in C_0^\infty (\RR), 
\qquad
\text{for all $j \geq 0$},
$$
and
$$
\text{supp } p_0 \subset (-\infty, 2],
\qquad
\text{supp } p_j \subset [2^{j-1}, 2^{j+1}],
\qquad
\text{for all $j \geq 1$}.
$$

\begin{proposition}\label{prop:G}
Let $w$ be the solution to the Klein-Gordon equation
$$
- \Box w + w = f,
\qquad
\big( w, \del_t w \big)(0) = (w_0, w_1),
$$
with $f = f(t, x)$ a sufficiently nice function.
Then for all $t \geq 0$, it holds
\be 
\aligned
&\langle t + |x| \rangle^{3/2} |w(t, x)|
\\
\lesssim
&\sum_{j\geq 0,\, |I| \leq 4} \sup_{0\leq s \leq t} p_j(s) \big\| \langle s+|x| \rangle \Gamma^I f(s, x) \big\|
+
\sum_{j\geq 0,\, |I| \leq 4} \big\| \langle |x| \rangle p_j (|x|) \Gamma^I w(0, x) \big\|
\endaligned
\ee

\end{proposition}

As a consequence, we have the following version of Proposition \ref{prop:G}.

\begin{proposition}\label{prop:G1}
With the same settings as Proposition \ref{prop:G}, 
\begin{itemize}
\item
let $\delta' > 0$ and assume 
$$
 \sum_{|I| \leq 4} \big\| \langle s+|x| \rangle \Gamma^I f(s, x) \big\|
\leq  C_f,
$$
then we have
\be 
\langle t + |x| \rangle^{3/2} |w(t, x)| 
\lesssim
C_f \langle t \rangle^{\delta'}
+
\sum_{|I| \leq 4} \big\| \langle |x| \rangle \Gamma^I w(0, x) \big\|;
\ee

\item
let $\delta' > 0$ and assume 
$$
 \sum_{|I| \leq 4} \big\| \langle s+|x| \rangle \Gamma^I f(s, x) \big\|
\leq  C_f \langle s \rangle^{-\delta'},
$$
then we have
\be 
\langle t + |x| \rangle^{3/2} |w(t, x)| 
\lesssim
C_f
+
\sum_{|I| \leq 4} \big\| \langle |x| \rangle \Gamma^I w(0, x) \big\|;
\ee

\end{itemize}

\end{proposition}


\subsection{$L^\infty-L^\infty$ estimates for wave equations}

We recall a type of $L^\infty-L^\infty$ estimates for wave components with derivatives, which was proved in \cite{Kubota-Yokoyama, Katayama-Kubo}, and was also applied in \cite{Katayama12a}.

\begin{lemma}\label{lem:Japan}
If $u$ is a smooth solution to the wave equation
$$
-\Box u = f,
\qquad
\big( u, \del_t u \big)(t_0) = (0, 0),
$$
then one has 
\bel{eq:L-infty}
\aligned
&\langle t+|x| \rangle^{-\rho} \langle |x| \rangle \langle t-|x| \rangle^\kappa |\del u(t, x) |
\\
\lesssim
&\sup_{\tau \in [t_0, t]}  \sup_{|y-x| \leq t-\tau} |y| \langle \tau + |y| \rangle^{\kappa - \rho + \mu} \langle t - |y| \rangle^{1-\mu} \sum_{|I| + |J| \leq 1} \big| \del^I \Omega^J f(\tau, y) \big|,
\endaligned
\ee
in which $\rho \geq 0$, $\kappa \geq 1$, and $\mu >0$.
\end{lemma}

Note that if we can show 
$$
|\del u | 
\lesssim \langle |x| \rangle^{-1} \langle t-|x| \rangle^{-1},
$$
then we obtain
$$
|\del u | 
\lesssim \langle t+|x| \rangle^{-1} \langle t-|x| \rangle^{-1},
\qquad
\text{for } |x| \geq t/2.
$$


\subsection{Estimates for commutators}

We will need to frequently use the following estimates for commutators, which can be found in \cite{Sogge, Hormander}.

\begin{lemma} \label{lem:comm2}
Let $u$ be a sufficiently nice function, then the following estimates are valid ($\Gamma, \Gamma', \Gamma'' \in V = \{\del_\alpha, \Omega_{ab}, L_a \}$)
\bel{eq:commu2} 
\aligned
\big| \Gamma \Box u \big|
&\lesssim
\big| \Box \Gamma u \big|,
\\
\big|\Gamma \Gamma' u \big|
&\lesssim
\big| \Gamma'\Gamma u \big| +  \big| \Gamma'' u \big|,
\\
\big|\Gamma L_0 u \big|
&\lesssim
\big| L_0 \Gamma u \big| +  \big| \Gamma' u \big|,
\endaligned
\ee
\end{lemma}



\section{Comparison between the ghost weight method and the hyperboloidal foliation method}\label{sec:Comparison}

This section is devoted to make a comparison between the ghost weight method and the hyperboloidal foliation method on how to take advantage of the $t-r$ decay, and is intended to provide some intuitions when dealing with different kinds of nonlinearities.

We first recall some notations used in the hyperboloidal foliation method of the version \cite{PLF-YM-book}, and all of the functions considered are assumed to be supported in the cone $\Kcal = \{ (t, x) : t \geq 2, t \geq |x| + 1 \}$ whenever we apply the hyperboloidal foliation method\footnote{We remind one that there exist two ways to remove the compacteness assumption, see \cite{PLF-YM-arXiv1, Klainerman-QW-SY}.}. A hyperboloid is denoted by $\Hcal_s = \{ (t, x) : t^2 = |x|^2 + s^2 \}$, and $s$ is called the hyperbolic time (with $s \geq 2$) . For a point $(t, x) \in \Hcal_s \bigcap \Kcal$, the following relations are frequently used
\be 
t \geq |x| + 1,
\qquad
t \leq t+r \leq 2t,
\qquad
s \leq t \leq s^2.
\ee
Also we use $\Kcal_{[s_0, s_1]} := \{(t, x): s_0^2 \leq t^2- r^2 \leq s_1^2; r\leq t-1 \}$ to denote subsets of $\Kcal$ which are limited by two hyperboloids $\Hcal_{s_0}$ and $\Hcal_{s_1}$ with $s_0 \leq s_1$.

Next, we introduce the semi-hyperboloidal frame \cite{PLF-YM-book} defined by
\bel{eq:semi-hyper}
\underdel_0:= \del_t, 
\qquad 
\underdel_a:= {L_a \over t} = {x^a\over t}\del_t+ \del_a.
\ee
On the other hand, the natural Cartesian frame can be expressed in terms of the semi-hyperboloidal frame as
\be 
\del_t = \underdel_0,
\qquad
\del_a = - {x^a \over t} \del_t + \underdel_a.
\ee

\subsection{Energy estimates on hyperboloids}

We revisit the energy estimates adapted to the hyperboloidal foliation setting, which allow us to bound the energy for the wave components and the Klein-Gordon components.

Let $\phi$ be a sufficiently nice function defined on a hyperboloid $\Hcal_s$, following \cite{PLF-YM-book} we define its natural energy $\widetilde{\Ecal}_m$ (with three equivalent expressions) by
\bel{eq:2energy} 
\aligned
\widetilde{\Ecal}_m(s, \phi)
&:=
 \int_{\Hcal_s} \Big( \big(\del_t \phi \big)^2+ \sum_a \big(\del_a \phi \big)^2+ 2 (x^a/t) \del_t \phi \del_a \phi + m^2 \phi ^2 \Big) \, dx
\\
               &= \int_{\Hcal_s} \Big( \big( (s/t)\del_t \phi \big)^2+ \sum_a \big(\underdel_a \phi \big)^2+ m^2 \phi^2 \Big) \, dx
                \\
               &= \int_{\Hcal_s} \Big( \big( \underdel_\perp \phi \big)^2+ \sum_a \big( (s/t)\del_a \phi \big)^2+ \sum_{a<b} \big( t^{-1}\Omega_{ab} \phi \big)^2+ m^2 \phi^2 \Big) \, dx,
 \endaligned
 \ee
in which $\underdel_{\perp} := L_0/t = \del_t+ (x^a / t) \del_a$ is the orthogonal vector field. 
The above integral is defined by
\bel{eq:flat-int}
\int_{\Hcal_s}|\phi | \, dx 
=\int_{\RR^3} \big|\phi(\sqrt{s^2+r^2}, x) \big| \, dx,
\ee
and we denote
\be 
\| \phi \|_{L^p_f(\Hcal_s)}
=
\Big( \int_{\Hcal_s} |\phi|^p \, dx \Big)^{1/p},
\qquad
1\leq p < +\infty.
\ee
Note that the second and the third expressions in \eqref{eq:2energy} yield
$$
\big\| (s/t) \del \phi \big\|_{L^2_f(\Hcal_s)} + \sum_a \big\| \underdel_a \phi \big\|_{L^2_f(\Hcal_s)}
\lesssim
\mathcal{E}_m(s, \phi)^{1/2}.
$$

Now, we demonstrate the energy estimates adapted to the hyperboloidal setting \cite{PLF-YM-book}, and one refers to \cite{PLF-YM-book} for the proof.

\begin{proposition}[Energy estimates for wave-Klein-Gordon equations]
For $m \geq 0$ let $\phi$ be the solution to the equation
$$
\aligned
-\Box \phi + m^2 \phi = h,
\qquad
\big( \phi, \del_t \phi \big)(t=2) = (\phi_0, \phi_1).
\endaligned
$$
Then for $s \geq 2$, it holds 
\bel{eq:w-EE} 
\widetilde{\Ecal}_m(s, \phi)^{1/2}
\leq 
\widetilde{\Ecal}_m(s_0, \phi)^{1/2}
+ \int_2^s \big\| h \big\|_{L^2_f(\Hcal_{s'})} \, ds'
\ee
for all sufficiently regular functions $\phi = \phi(t, x)$, which are defined and supported in $\Kcal_{[s_0, s]}$.
\end{proposition}


\subsection{Usage of the $t-r$ decay}

Before making the comparison between the ghost weight method and the hyperboloidal foliation method regarding the usage of the $t-r$ decay, we need to introduce the functions and the equations. Let $u, v, \phi$ be the solution to the equations (with $m = 0\, or\, 1$)
$$
\aligned
&-\Box u = 0,
\\
&-\Box v + v = 0,
\\
&-\Box \phi + m^2 \phi = Q(u, v),
\endaligned
$$
and this means that $u$ is a wave component, $v$ is a Klein-Gordon component, and $\phi$ could be a wave component or a Klein-Gordon component. For different nonlinear terms $Q(u, v)$, we want to estimate the energy $\widetilde{\Ecal}_m (s, \phi)^{1/2}$ using the hyperboloidal foliation method, and to estimate the energy $\Ecal_m (t, \phi)^{1/2}$ using the ghost weight method. 

\paragraph{The hyperboloidal foliation method}

For the linear wave component $u$ and the linear Klein-Gordon component $v$, we have
$$
\aligned
&\| (s/t) u \|
 + \| (s/t) \del u \|_{L^2_f(\Hcal_{s})} \lesssim 1,
\qquad\qquad
&| (s/t) u| + |(s/t) \del u | + \sum_a | \underdel_a u|
\lesssim t^{-3/2},
\\
&\| (s/t) \del v \|_{L^2_f(\Hcal_{s})} + \| v \|_{L^2_f(\Hcal_{s})} \lesssim 1,
\qquad\qquad
&|v| + |\del v | + \sum_a | \underdel_a v| \lesssim  t^{-3/2} \lesssim s^{-3/2}.
\endaligned
$$

Recall the energy estimates \eqref{eq:w-EE}, so in the following we will only estimate 
$$
\int_2^s \big\| Q(u, v) \big\|_{L^2_f(\Hcal_{s'})} \, ds'.
$$

We will need the following lemma on estimating null forms, whose proof can be found in \cite{PLF-YM-book} for instance.
\begin{lemma}
For the functions $u, v$, we have
\be 
\big| \del_\alpha u \del^\alpha v \big|
+
\sum_{\alpha, \beta} \big| \del_\alpha u \del_\beta v - \del_\alpha v \del_\beta u  \big|
\lesssim
(s/t)^2 \big| \del_t u \del_t v \big|
+
\sum_a \Big( \big| \del_t u \underdel_a v \big| + \big| \del_t v \underdel_a u \big| \Big)
+
\sum_{a, b} \big| \underdel_a u \underdel_b v \big|.
\ee

\end{lemma}

\begin{itemize}

\item Let $Q(u, v) = (\del v)^2$ (similarly for $Q(u, v) = v^2, v \del v$), then we have
$$
\aligned
\int_2^s \big\| Q(u, v) \big\|_{L^2_f(\Hcal_{s'})} \, ds'
\lesssim
\int_2^s \big\| (s'/t) \del v \big\|_{L^2_f(\Hcal_{s'})} \big\| (t/s') \del v \big\|_{L^\infty(\Hcal_{s'})} \, ds'
\lesssim
\int_2^s s'^{-3/2} \, ds'
\lesssim 1.
\endaligned
$$

\item Let $Q(u, v) = \del u v$ (similarly for $Q(u, v) = u v$), then we have
$$
\aligned
\int_2^s \big\| Q(u, v) \big\|_{L^2_f(\Hcal_{s'})} \, ds'
&\lesssim
\int_2^s \Big( \big\| (s'/t) \del u \big\|_{L^2_f(\Hcal_{s'})} \big\| (t/s')  v \big\|_{L^\infty(\Hcal_{s'})} + \big\| \del u \big\|_{L^\infty(\Hcal_{s'})} \big\|  v \big\|_{L^2_f(\Hcal_{s'})}  \Big)\, ds'
\\
&\lesssim
\int_2^s s'^{-3/2} \, ds'
\lesssim 1.
\endaligned
$$

\item Let $Q(u, v) = \del u \del v$ (similarly for $Q(u, v) = u \del v, (\del u)^2, u \del u, u^2$), then we have
$$
\aligned
&\quad
\int_2^s \big\| Q(u, v) \big\|_{L^2_f(\Hcal_{s'})} \, ds'
\\
&\lesssim
\int_2^s \Big( \big\| (s'/t) \del u \big\|_{L^2_f(\Hcal_{s'})} \big\| (t/s') \del v \big\|_{L^\infty(\Hcal_{s'})} + \big\| (t/s') \del u \big\|_{L^\infty(\Hcal_{s'})} \big\| (s'/t) \del v \big\|_{L^2_f(\Hcal_{s'})}  \Big)\, ds'
\\
&\lesssim
\int_2^s s'^{-1} \, ds'
\lesssim \log s.
\endaligned
$$

\item Let $Q(u, v) = \del_\alpha u \del^\alpha v$ (similarly for $Q(u, v) = \del_\alpha u \del^\alpha u, \del_\alpha u \del_\beta v - \del_\alpha v \del_\beta u$), then we have
$$
\aligned
&\quad
\int_2^s \big\| Q(u, v) \big\|_{L^2_f(\Hcal_{s'})} \, ds'
\\
&\lesssim
\int_2^s \Big( \big\| (s'/t) \del u \big\|_{L^2_f(\Hcal_{s'})} \big\| (s'/t) \del v \big\|_{L^\infty(\Hcal_{s'})} 
+ \big\| (s'/t) \del u \big\|_{L^\infty(\Hcal_{s'})} \big\| (s'/t) \del v \big\|_{L^2_f(\Hcal_{s'})}
\\
&\qquad+ \sum_a \big\| (t/s') \underdel_a u \big\|_{L^\infty(\Hcal_{s'})} \big\| (s'/t) \del v \big\|_{L^2_f(\Hcal_{s'})} 
+ \sum_a \big\| \underdel_a u \big\|_{L^2_f(\Hcal_{s'})} \big\| \del v \big\|_{L^\infty(\Hcal_{s'})} 
\\
&\qquad+ \sum_a \big\| (t/s') \underdel_a v \big\|_{L^\infty(\Hcal_{s'})} \big\| (s'/t) \del u \big\|_{L^2_f(\Hcal_{s'})} 
+ \sum_a \big\| \underdel_a v \big\|_{L^2_f(\Hcal_{s'})} \big\| \del u \big\|_{L^\infty(\Hcal_{s'})} 
\\
&\qquad+ \sum_{a, b} \big\| \underdel_a u \big\|_{L^2_f(\Hcal_{s'})} \big\| \underdel_b v \big\|_{L^\infty(\Hcal_{s'})} 
+ \sum_{a, b} \big\| \underdel_a u \big\|_{L^\infty(\Hcal_{s'})} \big\| \underdel_b v \big\|_{L^2_f(\Hcal_{s'})} 
 \Big)\, ds'
\\
&\lesssim
\int_2^s s'^{-3/2} \, ds'
\lesssim 1.
\endaligned
$$

\end{itemize}

\paragraph{The ghost weight method}

For the linear wave component $u$ and the linear Klein-Gordon component $v$, we have (with $G_a = (x_a/|x|) \del_t + \del_a$)
$$
\aligned
&\| u \| +\| \del u \| \lesssim 1,
\hskip4.5cm
|u| + |\del u | \lesssim \langle t+r \rangle^{-1} \langle t-r\rangle^{-1/2},
\\
&\| \del v \| + \| v \| \lesssim 1,
\hskip4.5cm
|v| + |\del v | \lesssim \langle t+r \rangle^{-3/2},
\\
&\sum_a \int_{t_0}^t \Big( \Big\| {G_a u\over \langle t'-r \rangle^{1/2+\delta}} \Big\|^2 + \Big\| {G_a v \over \langle t'-r \rangle^{1/2+\delta}} \Big\|^2 \Big) \, dt' 
+
\int_{t_0}^t \Big\| {v\over \langle t'-r \rangle^{1/2+\delta}} \Big\|^2 \, dt' 
\lesssim 1.
\endaligned
$$

Recall the energy estimates \eqref{eq:NEE}, so we will only estimate 
$$
\int_{t_0}^t \big\| Q(u, v) \big\| \, dt'.
$$

To apply the ghost weight method for null forms, we will rely on the following version of estimates on null forms.
\begin{lemma}
For the functions $u, v$, we have
\be 
\big| \del_\alpha u \del^\alpha v \big|
+
\sum_{\alpha, \beta} \big| \del_\alpha u \del_\beta v - \del_\alpha v \del_\beta u  \big|
\lesssim
\sum_a \Big(\big| G_a u \del v \big| + \big| G_a v \del u \big|\Big).
\ee

\end{lemma}
\begin{proof}
Recall that
$$
G_a = {x_a \over r} \del_t + \del_a = G^a,
$$
which gives us
$$
\del_a = G_a  - {x_a \over r} \del_t.
$$
Thus we have
$$
\aligned
\del_\alpha u \del^\alpha v
&=
-\del_t u \del_t v
+ \del_a u \del^a v
=
-\del_t u \del_t v
+ \big(  G_a u - {x_a \over r} \del_t u \big)   \big(  G^a v - {x^a \over r} \del_t v \big)
\\
&=
G_a u G^a v - G_a u {x^a \over r} \del_t v - {x_a \over r} \del_t u G^a v,
\endaligned
$$
as well as
$$
\aligned
\del_t u \del_a v - \del_t v \del_a u
&=
\del_t u \big( G_a v - {x_a \over r} \del_t v \big) 
-
\del_t v \big( G_a u - {x_a \over r} \del_t u \big) 
=
\del_t u G_a v
-
\del_t v G_a u,
\\
\del_a u \del_b v - \del_a v \del_b u
&=
\big( G_a u - {x_a \over r} \del_t u \big)   \big( G_b v - {x_b \over r} \del_t v \big) 
-
\big( G_a v - {x_a \over r} \del_t v \big)    \big( G_b u - {x_b \over r} \del_t u \big) 
\\
&=
G_a u G_b v - G_a u {x_b \over r} \del_t v - {x_a \over r} \del_t u G_b v
- G_a v G_b u + G_a v {x_b \over r} \del_t u + {x_a \over r} \del_t v G_b u.
\endaligned$$

The observation
$$
{|x_a| \over r} \leq 1,
\qquad
\big|G_a u \big| + \big| G_a v \big| \lesssim |\del u| + |\del v|
$$
concludes the desired result.
\end{proof}

\begin{itemize}

\item Let $Q(u, v) = (\del v)^2$ (similarly for $Q(u, v) = v^2, v \del v$), then we have
$$
\aligned
\int_{t_0}^t \big\| Q(u, v) \big\| \, dt'
\lesssim
\int_{t_0}^t \big\| \del v \big\| \big\| \del v \big\|_{L^\infty} \, dt'
\lesssim
\int_{t_0}^t \langle t' \rangle^{-3/2} \, dt'
\lesssim 1.
\endaligned
$$

\item Let $Q(u, v) = \del u v$ (similarly for $Q(u, v) = u v$), then we have
$$
\aligned
\int_{t_0}^t \big\| Q(u, v) \big\| \, dt'
&\lesssim
\int_{t_0}^t  \big\| \del u \big\| \big\| v \big\|_{L^\infty}  \, dt' 
+ \int_{t_0}^t  \big\| \langle t-r \rangle^{1/2+\delta} \del u \big\|_{L^\infty} \Big\| {v \over \langle t-r \rangle^{1/2+\delta}} \Big\| \, dt'
\\
&\lesssim
\int_{t_0}^t \langle t' \rangle^{-3/2} \, dt'
+
\Big(\int_{t_0}^t \langle t' \rangle^{-2 + 2\delta} \, dt' \Big)^{1/2} \Big(\int_{t_0}^t  \Big\| {v \over \langle t-r \rangle^{1/2+\delta}} \Big\|^2  \, dt' \Big)^{1/2}
\lesssim 1.
\endaligned
$$

\item Let $Q(u, v) = \del u \del v$ (similarly for $Q(u, v) = u \del v, (\del u)^2, u \del u, u^2$), then we have
$$
\aligned
\int_{t_0}^t \big\| Q(u, v) \big\| \, dt'
&\lesssim
\int_{t_0}^t  \big\| \del u \big\| \big\| \del v \big\|_{L^\infty}  \, dt' + \int_{t_0}^t  \big\| \del u \big\|_{L^\infty}  \big\|\del v \big\| \, dt'
\\
&\lesssim
\int_{t_0}^t \langle t' \rangle^{-1} \, dt'
\lesssim \log \langle t \rangle.
\endaligned
$$

\item Let $Q(u, v) = \del_\alpha u \del^\alpha v$ (similarly for $Q(u, v) = \del_\alpha u \del^\alpha u, \del_\alpha u \del_\beta v - \del_\alpha v \del_\beta u$), then we have
$$
\aligned
\int_{t_0}^t \big\| Q(u, v) \big\| \, dt'
&\lesssim
\sum_a \int_{t_0}^t  \Big\| {G_a u \over \langle t'-r \rangle^{1/2+\delta}} \Big\| \big\| \langle t'-r \rangle^{1/2+\delta} \del v \big\|_{L^\infty}  \, dt' 
\\
&\qquad
+ \int_{t_0}^t  \big\| \langle t'-r \rangle^{1/2+\delta} \del u \big\|_{L^\infty} \Big\| {G_a v \over \langle t'-r \rangle^{1/2+\delta}} \Big\| \, dt'
\\
&\lesssim
\Big( \int_{t_0}^t \langle t' \rangle^{-2+2\delta} \, dt' \Big)^{1/2}
\lesssim 1.
\endaligned
$$

\end{itemize}

From the above calculations for a few types of nonlinearities, we find that the spacetime integral is either uniformly bounded in both methods or has the same logarithmic growth in both methods. One advantage in the ghost weight method is that one does not need the solutions to have compact support.



\section{Contraction mapping argument}
\label{sec:Contraction}

\subsection{The solution space}\label{subsec:BA}


Given a pair of functions $(\Psi, \phi)$, with $\Psi$ taking values in $\RR^3$ and $\phi$ taking values in $\RR$, we define its $X$-norm by
\bel{eq:X-norm}
\aligned
\big\| (\Psi, \phi) \big\|_X
&:=
\sum_{|I| \leq N+1} \Big(\Ecal_{gst, 1} (t, \Gamma^I \Psi)^{1/2} 
+ \big\| \Gamma^I \phi \big\| \Big)
+\sum_{|I| \leq N-3 } \langle t+|x| \rangle^{3/2-\delta} \big| \Gamma^I \Psi \big|
\\
&+ \sum_{|I| \leq N-7} \sup_{t\geq t_0, \, x} \langle t+|x| \rangle^{3/2} | \Gamma^I \Psi |
+ \sum_{|I| \leq N-9} \sup_{t\geq t_0, \, x} \langle t+|x| \rangle \langle t-|x| \rangle^{1/2} | \Gamma^I \phi |,
\endaligned
\ee
in which $N\geq 15$ and $0 <\delta \ll 1/10$.

The solution space $X$ is defined as follows. 
\begin{definition}
A pair of functions $(\Psi, \phi)$ defined in $[t_0, +\infty) \times \RR^3$, with $\Psi$ taking values in $\RR^3$ and $\phi$ taking values in $\RR$, is said to lie in space $X$, if the pair of functions satisfy
\begin{itemize}

\item
\be 
\big(\Psi, \del_t \Psi, \phi, \del_t \phi \big)(t_0)
= (E_0, E_1, n_0, n_1).
\ee
\item 
\be
\big\| (\Psi, \phi) \big\|_X
\leq C_1 \eps,
\ee
with $C_1 \gg 1$ to be determined, and $\eps$ the size of the initial data.

\end{itemize}

\end{definition}

\subsection{The contraction mapping}

\begin{definition}
The solution mapping $T$ maps a pair of functions $(\Psi, \phi) \in X$ to the unique pair of functions $\big(\widetilde{\Phi}, \widetilde{\phi}\big)$, satisfying
\bel{eq:solution-map}
\aligned
&-\Box \widetilde{\Phi} + \widetilde{\Phi} = - \phi \Psi,
\\
&-\Box \widetilde{\phi} = \Delta |\Psi|^2,
\\
&\big(\widetilde{\Phi}, \del_t \widetilde{\Phi}, \widetilde{\phi}, \del_t \widetilde{\phi} \big)(t_0)
=
(E_0, E_1, n_0, n_1),
\endaligned
\ee
and we will denote $\big(\widetilde{\Psi}, \widetilde{\phi}\big) = T (\Psi, \phi)$.
\end{definition}

We want to show that the solution mapping $T$ maps an element in the solution space $X$ into $X$, and is a contraction mapping.
\begin{proposition}\label{prop:contraction001}
For two elements $(\Psi, \phi), (\Psi', \phi') \in X$, we denote $\big(\widetilde{\Psi}, \widetilde{\phi}\big) = T (\Psi, \phi), \big(\widetilde{\Psi'}, \widetilde{\phi'}\big) = T (\Psi', \phi')$, then we have
\bel{eq:refined}
\aligned
\big(\widetilde{\Psi}, \widetilde{\phi}\big), \big(\widetilde{\Psi'}, \widetilde{\phi'}\big) & \in X,
\\
\big\| \big( \widetilde{\Psi} - \widetilde{\Psi'}, \widetilde{\phi} - \widetilde{\phi'} \big)\big\|_X
&\leq {1\over 2} \big\| (\Psi - \Psi', \phi - \phi') \big\|_X.
\endaligned
\ee

\end{proposition}

We first explore some properties enjoyed by the elements in the solution space $X$.

\begin{lemma}
Let $(\Psi, \phi) \in X$, then we have
\be 
\aligned
\int_{t_0}^t \Big\| {\Gamma^I \Psi \over \langle t'-|x| \rangle^{1/2+\delta}} \Big\|^2 \, dt'
\lesssim (C_1 \eps)^2,
\qquad
|I| \leq N+1,
\\
\big\| \del \Gamma^I \Psi \big\| + \big\| \Gamma^I \Psi \big\| +  \big\| \Gamma^I \phi \big\|
\lesssim C_1 \eps,
\qquad
|I| \leq N+1,
\\
\big| \Gamma^I \Psi(t, x) \big|
\lesssim C_1 \eps \langle t+|x| \rangle^{-3/2+\delta},
\qquad
|I| \leq N-3,
\\
\big| \Gamma^I \Psi(t, x) \big|
\lesssim C_1 \eps \langle t+|x| \rangle^{-3/2},
\qquad
|I| \leq N-7,
\\
\big| \Gamma^I \phi(t, x) \big|
\lesssim C_1 \eps \langle t+|x| \rangle^{-1} \langle t-|x| \rangle^{-1/2},
\qquad
|I| \leq N-9,
\\
\big| \Gamma^I \Psi(t, x) \big| + \big| \Gamma^I \phi(t, x) \big|
\lesssim C_1 \eps \langle t+|x| \rangle^{-1},
\qquad
|I| \leq N-2.
\endaligned
\ee

\end{lemma}
\begin{proof}
Except the last estimates, the proof for others follows from the definition of the solution space $X$ and the ghost weight energy $\Ecal_{gst, m}(t, \Psi)$.

As for the last estimates, it follows from the second estimates and the Klainerman-Sobolev inequality in Proposition \ref{prop:K-S}.
\end{proof}

Let $(\Psi, \phi) \in X$, and denote $\big(\widetilde{\Psi}, \widetilde{\phi}\big) = T (\Psi, \phi)$.
We decompose the wave component $\widetilde{\phi}$, and we reformulate the equations as
\bel{eq:reform001}
\aligned
&-\Box \widetilde{\Phi} + \widetilde{\Phi} = - \phi \Psi,
\\
&-\Box  \widetilde{\phi}^0 = 0,
\qquad
-\Box  \widetilde{\phi}^\Delta = |\Psi|^2, 
\\
&\big(\widetilde{\Phi}, \del_t \widetilde{\Phi}, \widetilde{\phi}^0, \del_t \widetilde{\phi}^0,  \widetilde{\phi}^\Delta, \del_t \widetilde{\phi}^\Delta \big)(t_0)
=
(E_0, E_1, n_0, n_1, 0, 0),
\endaligned
\ee
with the relation $\widetilde{\phi} = \widetilde{\phi}^0 + \Delta \widetilde{\phi}^\Delta$. 
We act the vector fields $\Gamma^I, \Gamma^J, \del \Gamma^I$, with $|I| \leq N+1, |J| \leq N$, to the reformulated equations to get
\bel{eq:reform002}
\aligned
&-\Box \Gamma^I \widetilde{\Phi} + \Gamma^I \widetilde{\Phi} = - \Gamma^I \big(\phi \Psi \big),
\\
&-\Box \Gamma^J \widetilde{\phi}^0 = 0,
\qquad
-\Box \del \Gamma^I \widetilde{\phi}^\Delta = \del \Gamma^I |\Psi|^2, 
\\
&\big(\widetilde{\Phi}, \del_t \widetilde{\Phi}, \widetilde{\phi}^0, \del_t \widetilde{\phi}^0,  \widetilde{\phi}^\Delta, \del_t \widetilde{\phi}^\Delta \big)(t_0)
=
(E_0, E_1, n_0, n_1, 0, 0).
\endaligned
\ee

\begin{lemma}[Conformal energy estimates for $\widetilde{\phi}^\Delta$ component]
We have
\bel{eq:wave-con}
E_{con} (t, \del \Gamma^I \widetilde{\phi}^\Delta )^{1/2}
\lesssim \eps + (C_1 \eps)^2 \langle t \rangle^{1/2},
\qquad
|I| \leq N+1.
\ee
\end{lemma}
\begin{proof}
The proof is straightforward. First act $\del \Gamma^I$, with $|I| \leq N+1$, to the $\widetilde{\phi}^\Delta$ equation to get
$$
-\Box \del \Gamma^I \widetilde{\phi}^\Delta = \del \Gamma^I |\Psi|^2.
$$
We then apply the conformal energy estimates in \eqref{eq:conformal-EE} to find (recall $N \geq 15$)
$$
\aligned
E_{con} (t, \del \Gamma^I \widetilde{\phi}^\Delta)^{1/2}
&\lesssim
E_{con} (t_0, \del \Gamma^I \widetilde{\phi}^\Delta)^{1/2}
+
\int_{t_0}^t \big\| \langle t'+|x| \rangle \del \Gamma^I \big( |\Psi|^2 \big) \big\| \,  dt'
\\
&\lesssim
\eps 
+ 
\sum_{|I_1| \leq N-7, |I_2| \leq N+1} \int_{t_0}^t \big\| \langle t'+|x| \rangle \Gamma^{I_1} \Psi \big\|_{L^\infty} \big\| \del \Gamma^{I_2} \Psi \big\| \, dt'
\\
&+ 
\sum_{|I_1|\leq N-7, |I_2| \leq N+1} \int_{t_0}^t \big\| \langle t'+|x| \rangle \Gamma^{I_1} \Psi \big\|_{L^\infty} \big\| \Gamma^{I_2} \Psi \big\| \, dt'
\\
&\lesssim
\eps + (C_1 \eps)^2 \langle t \rangle^{1/2}.
\endaligned
$$

The proof is done.
\end{proof}

As a consequence, we are able to obtain good pointwise estimates for $\widetilde{\phi}$ away from the light cone by the aid of the scaling vector field, and this strategy of taking advantage of the scaling vector field in the coupled wave and Klein-Gordon equations was first used in \cite{Dong1912}.
\begin{lemma}\label{lem:sup-wave}
Let $(\Psi, \phi) \in X$, and denote $\big(\widetilde{\Psi}, \widetilde{\phi}\big) = T (\Psi, \phi)$, then we have
\bel{eq:pointwise-in}
\big| \Gamma^I \widetilde{\phi} |
\lesssim
\big(\eps + (C_1 \eps)^2\big) \langle t+|x| \rangle^{-1/2} \langle t-|x| \rangle^{-1},
\qquad
|I| \leq N-2.
\ee 
\end{lemma}
\begin{proof}
The conformal energy estimates in \eqref{eq:wave-con} give
$$
\big\| L_0 \del \Gamma^I \widetilde{\phi}^\Delta \big\|
\lesssim \big(\eps + (C_1 \eps)^2\big) \langle t \rangle^{1/2},
\qquad
|I| \leq N+1,
$$
which implies
$$
\sum_{|I_1| \leq 3, |I_2| \leq N-2} \big\| \Gamma^{I_1} L_0 \del \Gamma^{I_2} \widetilde{\phi}^\Delta \big\|
\lesssim \big(\eps + (C_1 \eps)^2\big) \langle t \rangle^{1/2}.
$$
The application of the Klainerman-Sobolev inequality \eqref{eq:K-S} deduces 
$$
\big| L_0 \del \Gamma^{I} \widetilde{\phi}^\Delta \big|
\lesssim
\big(\eps + (C_1 \eps)^2\big) \langle t+|x| \rangle^{-1/2},
\qquad
|I| \leq N-2.
$$
Similarly, we have
$$
\big| \Gamma \del \Gamma^{I} \widetilde{\phi}^\Delta \big|
\lesssim
\big(\eps + (C_1 \eps)^2\big) \langle t+|x| \rangle^{-1/2},
\qquad
|I| \leq N-2.
$$

Then, we recall the simple relation
$$
|\del w | \lesssim \langle t-|x| \rangle^{-1} \big( |L_0 w| + |\Gamma w| \big),
$$
which allows us to have
$$
\big| \del \del \Gamma^{I} \widetilde{\phi}^\Delta \big|
\lesssim
\big(\eps + (C_1 \eps)^2\big) \langle t+|x| \rangle^{-1/2} \langle t-|x| \rangle^{-1},
\qquad
|I| \leq N-2.
$$
Finally, the relation $\widetilde{\phi} = \widetilde{\phi}^0 + \Delta \widetilde{\phi}^\Delta$ implies
$$
\aligned
\big| \Gamma^I \widetilde{\phi} \big|
\lesssim
\big| \Gamma^I \widetilde{\phi}^0 \big| + \big| \Gamma^I \Delta \widetilde{\phi}^\Delta \big|
\lesssim
\big( \eps + (C_1 \eps)^2 \big) \langle t+|x| \rangle^{-1/2} \langle t-|x| \rangle^{-1},
\qquad
|I| \leq N-2.
\endaligned
$$

The proof is complete.
\end{proof}

As a complement of Lemma \ref{lem:sup-wave}, we have the following result, which gives us good pointwise decay of $\widetilde{\phi}$ away from the origin.
\begin{lemma}\label{lem:sup-wave2}
It holds
\bel{eq:pointwise-ext}
\big| \Gamma^I \widetilde{\phi} \big|
\lesssim
\big( \eps + (C_1 \eps)^2 \big) \langle |x| \rangle^{-1} \langle t-|x| \rangle^{-1/2},
\qquad
|I| \leq N-9.
\ee 

\end{lemma}

\begin{proof}
Consider the $\widetilde{\phi}^\Delta$ equation in \eqref{eq:reform002}, and note that we have
$$
\sum_{|I| \leq N-9, |J_1| + |J_2| \leq 1} \big| \del^{J_1} \Omega^{J_2} \big( \del \Gamma^I |\Psi|^2 \big) \big|
\lesssim (C_1 \eps)^2 \langle t+|x| \rangle^{-3}.
$$
Then we apply Lemma \ref{lem:Japan}, with $\rho=0, \kappa=1, \mu = 1/2$, to obtain
$$
\big| \del \del \Gamma^I \widetilde{\phi}^\Delta \big|
\lesssim (C_1 \eps)^2 \langle |x| \rangle^{-1} \langle t-|x| \rangle^{-1},
\qquad
|I| \leq N-9.
$$
Finally, we recall that
$$
\big| \Gamma^I \widetilde{\phi} \big|
\lesssim
\big| \Gamma^I \widetilde{\phi}^0 \big|
+
\big| \Gamma^I \Delta \widetilde{\phi}^\Delta \big| 
\lesssim
\big( \eps + (C_1 \eps)^2 \big) \langle |x| \rangle^{-1} \langle t-|x| \rangle^{-1/2},
\qquad
|I| \leq N-9.
$$

\end{proof}

A combination of Lemma \ref{lem:sup-wave} and Lemma \ref{lem:sup-wave2} gives us the following pointwise decay result of $\widetilde{\phi}$ component.
\begin{proposition}\label{prop:wave-decay}
We have
\bel{eq:pointwise-all}
\big| \Gamma^I \widetilde{\phi} \big|
\lesssim
\big( \eps + (C_1 \eps)^2 \big) \langle  t + |x| \rangle^{-1} \langle t-|x| \rangle^{-1/2},
\qquad
|I| \leq N-9.
\ee 

\end{proposition}

\begin{proposition}\label{prop:contraction002}
Let $(\Psi, \phi) \in X$, and denote $\big(\widetilde{\Psi}, \widetilde{\phi}\big) = T (\Psi, \phi)$, then we have
\bel{eq:refine001}
\aligned
\big\| \big( \widetilde{\Psi}, \widetilde{\phi} \big) \big\|_X
\lesssim \eps + (C_1 \eps)^{3/2}.
\endaligned
\ee
\end{proposition}

\begin{proof}
Recall the definition of the $X$-norm in \eqref{eq:X-norm}, we need to show that each of terms has the desired bound. We act the vector fields $\Gamma^I, \Gamma^J, \del \Gamma^I$, with $|I| \leq N+1, |J| \leq N$, to the reformulated equations to get
$$
\aligned
&-\Box \Gamma^I \widetilde{\Phi} + \Gamma^I \widetilde{\Phi} = - \Gamma^I \big(\phi \Psi \big),
\\
&-\Box \Gamma^J \widetilde{\phi}^0 = 0,
\qquad
-\Box \del \Gamma^I \widetilde{\phi}^\Delta = \del \Gamma^I |\Psi|^2, 
\\
&\big(\widetilde{\Phi}, \del_t \widetilde{\Phi}, \widetilde{\phi}^0, \del_t \widetilde{\phi}^0,  \widetilde{\phi}^\Delta, \del_t \widetilde{\phi}^\Delta \big)(t_0)
=
(E_0, E_1, n_0, n_1, 0, 0),
\endaligned
$$
with the relation $\widetilde{\phi} = \widetilde{\phi}^0 + \Delta \widetilde{\phi}^\Delta$.

\textbf{Step 1.}
First, we show the bound for ghost energy for the $\widetilde{\Psi}$ component.
By the energy estimates \eqref{eq:GEE}, we have
$$
\aligned
\Ecal_{gst, 1} (t, \Gamma^I \widetilde{\Psi})
&\lesssim
\Ecal_{gst, 1} (t_0, \Gamma^I \widetilde{\Psi})
+
\int_{t_0}^t \big\| \del_t \Gamma^I \Psi \Gamma^I \big(\phi \Psi \big) \big\|_{L^1} \, dt'
\\
&\lesssim 
\eps^2
+
\int_{t_0}^t \big\| \del_t \Gamma^I \Psi \big\| \big\| \Gamma^I \big(\phi \Psi \big) \big\| \, dt'
\\
&\lesssim 
\eps^2
+
C_1 \eps \int_{t_0}^t \big\| \Gamma^I \big(\phi \Psi \big) \big\| \, dt'.
\endaligned
$$
We proceed to get
$$
\aligned
\Ecal_{gst, 1} (t, \Gamma^I \widetilde{\Psi})
&\lesssim 
\eps^2
+
C_1 \eps \int_{t_0}^t \big\| \Gamma^I \big(\phi \Psi \big) \big\| \, dt'
\\
&\lesssim
\epsilon^2
+
C_1 \eps \sum_{|I_1| + |I_2| \leq N+1} \int_{t_0}^t \big\| \Gamma^{I_1} \phi \Gamma^{I_2} \Psi \big\| \, dt'
\\
&\lesssim
\epsilon^2
+
C_1 \eps \sum_{|I_1|\leq N-9, |I_2| \leq N+1} \int_{t_0}^t \big\| \langle t'-|x| \rangle^{1/2+ \delta} \Gamma^{I_1} \phi \big\|_{L^\infty}  \Big\| {\Gamma^{I_2} \Psi \over \langle t'-|x| \rangle^{1/2+\delta} } \Big\| \, dt'
\\
&+
C_1 \eps \sum_{|I_1|\leq N+1, |I_2| \leq N-3} \int_{t_0}^t \big\| \Gamma^{I_1} \phi \big\|  \big\| \Gamma^{I_2} \Psi \big\|_{L^\infty} \, dt'
\\
&\lesssim
\eps^2 
+
(C_1 \eps)^3 \Big( \int_{t_0}^t \langle t' \rangle^{-2+2\delta} \, dt' \Big)^{1/2} 
+
(C_1 \eps)^3 \int_{t_0}^t \langle t' \rangle^{-3/2+\delta} \, dt'
\lesssim
\eps^2 + (C_1 \eps)^3.
\endaligned
$$

\textbf{Step 2.} We now show the pointwise estimates for the $\widetilde{\Psi}$ component.

On one hand, for $|I| \leq N+1$ we have
$$
\aligned
&\quad
\big\| \langle t+|x| \rangle \Gamma^I (\phi \Psi) \big\|
\\
&\lesssim
\sum_{|I_1|\leq N-3, |I_2| \leq N+1 } \big\| \langle t+|x| \rangle \Gamma^{I_1} \Psi \big\|_{L^\infty}  \big\| \Gamma^{I_2} \phi \big\|
+
\sum_{|I_1|\leq N+1, |I_2| \leq N-2 } \big\| \Gamma^{I_1} \Psi \big\|  \big\| \langle t+|x| \rangle \Gamma^{I_2} \phi \big\|_{L^\infty}
\\
&\lesssim (C_1 \eps)^2.
\endaligned
$$
Then the application of the Sobolev inequality in Proposition \ref{prop:G1} deduces that
$$
\big| \Gamma^I \widetilde{\Psi} \big|
\lesssim \big( \eps + (C_1 \eps)^2 \big) \langle t+|x| \rangle^{-3/2+\delta},
\qquad
|I| \leq N-3.
$$

With $|I| \leq N-3$, we note that
$$
\aligned
\big\| \langle t+|x| \rangle \Gamma^I (\phi \Psi) \big\|
\lesssim
\sum_{|I_1|, |I_2| \leq N-3 } \big\| \langle t+|x| \rangle \Gamma^{I_1} \Psi \big\|_{L^\infty}  \big\| \Gamma^{I_2} \phi \big\|
\lesssim (C_1 \eps)^2 \langle t \rangle^{-\delta}.
\endaligned
$$
Thus the Sobolev inequality in Proposition \ref{prop:G1} implies that
$$
\big| \Gamma^I \widetilde{\Psi} \big|
\lesssim \big( \eps + (C_1 \eps)^2 \big) \langle t+|x| \rangle^{-3/2},
\qquad
|I| \leq N-7.
$$

\textbf{Step 3.} We turn to derive the bounds for the $L^2$ norms of $\widetilde{\phi}$ component.

First, for the $\widetilde{\phi}^0$ part, we easily get
$$
\aligned
\Ecal (t, \Gamma^J \widetilde{\phi}^0)^{1/2}
+
\Ecal_{con} (t, \Gamma^J \widetilde{\phi}^0)^{1/2}
\leq
\Ecal (t_0, \Gamma^J \widetilde{\phi}^0)^{1/2}
+
\Ecal_{con} (t_0, \Gamma^J \widetilde{\phi}^0)^{1/2}
\lesssim
\eps,
\qquad
|J| \leq N,
\endaligned
$$
which gives
$$
\big\| \Gamma^I \widetilde{\phi}^0 \big\|
\lesssim
\eps,
\qquad
|I| \leq N+1.
$$

Next, we treat the $\widetilde{\phi}^\Delta$ part, and the energy estimates imply
$$
\aligned
\Ecal (t, \del \Gamma^I \widetilde{\phi}^\Delta)^{1/2}
&\lesssim
\Ecal (t_0, \del \Gamma^I \widetilde{\phi}^\Delta)^{1/2}
+
\int_{t_0}^t \big\| \del \Gamma^I |\Psi|^2 \big\| \, dt'
\\
&\lesssim
\eps + (C_1 \eps)^2 \int_{t_0}^t \langle t' \rangle^{-3/2} \, dt'
\lesssim
\eps + (C_1 \eps)^2,
\qquad
|I| \leq N+1,
\endaligned
$$
which means
$$
\big\| \del \del \Gamma^I \widetilde{\phi}^\Delta  \big\|
\lesssim
\eps + (C_1 \eps)^2,
\qquad
|I| \leq N+1.
$$

Gathering the above estimates yields
$$
\big\| \Gamma^I \widetilde{\phi} \big\|
\lesssim
\eps + (C_1 \eps)^2,
\qquad
|I| \leq N+1.
$$

\textbf{Step 4.}
Finally, by recalling the estimates in Proposition \ref{prop:wave-decay}, we find
$$
\big\| \big( \widetilde{\Psi}, \widetilde{\phi} \big) \big\|_X
\lesssim \eps + (C_1 \eps)^{3/2}.
$$
Thus we finish the proof.
\end{proof}

In the same way as the proofs for Proposition \ref{prop:wave-decay} and Proposition \ref{prop:contraction002}, we have the following result.
\begin{proposition}\label{prop:contraction003}
For two elements $(\Psi, \phi), (\Psi', \phi') \in X$, we denote $\big(\widetilde{\Psi}, \widetilde{\phi}\big) = T (\Psi, \phi), \big(\widetilde{\Psi'}, \widetilde{\phi'}\big) = T (\Psi', \phi')$, then we have
\bel{eq:refined}
\aligned
\big\| \big( \widetilde{\Psi} - \widetilde{\Psi'}, \widetilde{\phi} - \widetilde{\phi'} \big)\big\|_X
&\lesssim (C_1 \eps)^{1/2} \big\| (\Psi - \Psi', \phi - \phi') \big\|_X.
\endaligned
\ee

\end{proposition}

Thus we are now able to provide the proof for Proposition \ref{prop:contraction001}.

\begin{proof}[Proof of Proposition \ref{prop:contraction001}]
We choose $C_1 \gg 1$ very large, and $\eps \ll 1$ sufficiently small, so that the estimates in Proposition \ref{prop:contraction002} and Proposition \ref{prop:contraction003} lead us to
$$
\aligned
\big\| \big(\widetilde{\Psi}, \widetilde{\phi}\big) \big\|_X
\leq {1\over 2} C_1 \eps,
\\
\big\| \big( \widetilde{\Psi} - \widetilde{\Psi'}, \widetilde{\phi} - \widetilde{\phi'} \big)\big\|_X
&\leq {1\over 2} \big\| (\Psi - \Psi', \phi - \phi') \big\|_X.
\endaligned
$$

The proof is done.
\end{proof}

Finally, we provide the proof for the main theorem.
\begin{proof}[Proof of Theorem \ref{thm:main1}]
By the fixed point theorem, we know there exists a unique element $(E, n) \in X$, such that
$$
(E, n) = T(E, n).
$$
This means that $(E, n)$ is the solution to the original Klein-Gordon-Zakharov equations in \eqref{eq:model-KGZ}. According to the definition of the solution space $X$, we know that the sharp pointwise decay results \eqref{eq:thm-decay} are valid, and the energy for $(E, n)$ is uniformly bounded.

\end{proof}



%



{\footnotesize

\end{document}